\documentclass[12pt]{amsart}
\usepackage[latin1]{inputenc}
\usepackage{amsxtra,amssymb,amsthm,amsmath,amscd,mathrsfs, epsfig,eufrak,url}
\usepackage[all]{xy}
\def \N {{\mathbb N}}

\def\le{\leqslant}
\def\ge{\geqslant}

\theoremstyle{plain}
\newtheorem{theorem}{Theorem}

\newtheorem{lemma}{Lemma}[section]
\newtheorem{corollary}[theorem]{Corollary}

\newtheorem*{problem*}{Problem}

\theoremstyle{remark}
\newtheorem{remark}[theorem]{Remark}

\theoremstyle{definition}

\numberwithin{equation}{section}

\begin{document}

\title[Egyptian Fractions with Restrictions]{Egyptian
Fractions with Restrictions}
\date{\today}
\author{Yong-Gao Chen, Christian Elsholtz and Li-Li Jiang}

\address{School of Mathematical Sciences and Institute of Mathematics, Nanjing
Normal University\\
 Nanjing 210046, P. R. CHINA}
\email{ ygchen@njnu.edu.cn}

 \address{Institut f\"ur Mathematik A, Steyrergasse 30/II,
A-8010 Graz, Austria } \email{ elsholtz@math.tugraz.at}

\address{ School of Mathematical Sciences, Nanjing
Normal University\\
 Nanjing 210046, P. R. CHINA }

\begin{abstract} Let $T_o(k)$ denote
the number of solutions of $\sum_{i=1}^k\frac 1{x_i}=1$ in odd
numbers $1<x_1<x_2<\cdots <x_k$. It is clear that $T_o(2k)=0$. For
distinct primes $p_1, p_2,\ldots , p_t$, let $S(p_1, p_2,\ldots ,
p_t)=\{ p_1^{\alpha_1}\cdots p_t^{\alpha_t}\mid \alpha_i\in
\mathbb{N}_0, i=1,2,\ldots , t \}$. Let $T_k(p_1,\ldots , p_t)$ be
the number of solutions $\sum_{i=1}^{k}\frac 1{x_i}=1$ with
$1<x_1<x_2<\cdots <x_{k}$ and $x_i\in S(p_1, p_2,\ldots , p_t)$.
It is clear that if $T_k(p_1,\ldots , p_t)\not= 0$ for some $k$,
then the inverse sum of all elements $s_j>1$ in $S(p_1, p_2,\ldots, p_t)$
is more than 1.

In this paper we study $T_o(k)$ and $T_k(p_1,\ldots , p_t)$. Three
of our results are:\\
1) $T_o(2k+1)\ge (\sqrt 2)^{(k+1)(k-4)}$ for
all $k\ge 4$;\\
 2) if the inverse sum of all elements $s_j>1$ in
$S(p_1, p_2,\ldots , p_t)$ is more than 1, then $T_k(p_1,\ldots ,
p_t)\not= 0$ for infinitely many  $k$ and the set of these $k$ is
the union of finitely many arithmetic progressions;\\
3) there exists two
constants $k_0=k_0(p_1,\ldots , p_t)>1$ and $c=c(p_1,\ldots ,
p_t)>1$ such that for any $k>k_0$ we have either $T_k(p_1,\ldots ,
p_t)= 0$ or $T_k(p_1,\ldots , p_t)>c^k$.
\end{abstract}


\thanks{The work of Y.-G. Chen is supported by the National Natural Science Foundation of China,
Grant Nos 11071121 and 10771103.}

\maketitle

\section{Introduction}

Egyptian fractions or unit fractions are extensively studied (see
\cite{Barbeau}, \cite{Croot}, \cite[D11]{Guy}, \cite{Martin}).
Some studies are concerned with the question which fractions can be written as
a sum of $k$ unit fractions, others restrict the denominators, still others
count the number of solutions. In particular, solutions of
 $1=\sum_{i=1}^k\frac{1}{x_i}$ have been extensively studied.
Sierpi\'{n}ski \cite{Sierpinski} noted that there is a solution
with distinct odd integers, and Breusch \cite{StarkeandBreusch}
and Stewart \cite{Stewart} independently proved that each fraction
$\frac{a}{b}$ with odd denominator can be written as a finite sum
of distinct unit fractions. More recently Shiu \cite{Shiu} and
 Burshtein \cite{Burshtein} proved that the equation
$\sum_{i=1}^9\frac 1{x_i}=1$ in distinct odd numbers has only
those five solutions that can be easily found with a computer.
 Motivated by this, let $T_o(k)$ denote the
number of solutions of $\sum_{i=1}^k\frac 1{x_i}=1$ in odd numbers
$1<x_1<x_2<\cdots <x_k$. It is easy to see that for all even
values of $k$, the equation $\sum_{i=1}^k\frac 1{x_i}=1$ does not have
solutions in odd numbers and then $T_o(k)=0$.  One natural problem
is: how large can $T_o(k)$ be, for odd $k$? In this paper we present a
 lower bound which grows faster than exponentially.

The literature contains many results either stating that there are
solutions of $\sum_{i=1}^k\frac 1{x_i}=1$ of a special type, which
is an indication that the equation has many solutions, or stating
that certain types of solutions cannot exist, or bounding the number
of solutions. For example, Martin \cite{Martin} showed that
$\sum_{i=1}^k\frac 1{x_i}=1$ has solutions, in which a dense set of
the possible denominators occur. Croot \cite{Croot} showed that for
any $r$-colouring of the integers there is a monochromatic solution
of $\sum_{i=1}^k\frac 1{x_i}=1$. This is some measure of saying the
equation has many solutions, and these are closely interlinked, as
otherwise one could construct a bad colouring.

In 2007 Z.W.~Sun \cite{Sun} conjectured the following
strengthening of this: If $A \subset \mathbb N$ is a set of
positive upper asymptotic density, then there is a finite subset
$\{x_1, \ldots, x_k\}$ of $A$ with $\sum_{i=1}^k \frac{1}{x_i}=1$.

In this paper we examine for which set of primes, there is a
solution of the diophantine equation $\sum_{i=1}^k
\frac{1}{x_i}=1$, of which all denominators consist of the given
prime factors only, and how many of such solutions there are. Let
us introduce the following notation. Let $\mathbb{N}_0$ be the set
of all nonnegative integers.  For distinct primes $p_1, p_2,\ldots
, p_t$, let
$$S(p_1, p_2,\ldots , p_t)=\{ p_1^{\alpha_1}\cdots p_t^{\alpha_t}\mid
\alpha_i\in \mathbb{N}_0,  i=1,2,\ldots , t \}$$ and let
$T_k(p_1,\ldots , p_t)$ be the number of solutions of
$\sum_{i=1}^{k}\frac 1{x_i}=1$ with $1<x_1<x_2<\cdots <x_{k}$ and
$x_i\in S(p_1, p_2,\ldots , p_t)\ (1\le i\le k)$.

As a very special case Burshtein \cite{Burshtein1} proved that the
equation $\sum_{i=1}^{11}\frac 1{x_i}=1$   with $1<x_1<x_2<\cdots
<x_{11}$ and $x_i\in \{ 3^{\alpha} 5^{\beta} 7^{\gamma} : \alpha,
\beta,\gamma \in \N_0 \},
 (1\le i\le 11)$  has exactly
17 solutions, in other words $T_{11}(3,5,7)=17$.

In this paper we establish a necessary and sufficient condition on the set of
primes for a solution to exist, and present upper and lower bounds of
exponential type. (For details see the next section).

There is a closely related problem, where not all denominators are
necessarily distinct. Let us also review results that are known
for counting such solutions. Let $K(k)$ denote the number of
solutions of $\sum_{i=1}^k\frac 1{x_i}=1$ in integers $1\le x_1\le
x_2\le \cdots \le x_k$. Erd\H os, Graham and Straus (unpublished
but see \cite[p.32]{erdos}) proved that
$$e^{k^{2-\varepsilon }} < K(k)< c_0^{2^k},$$
where $c_0=1.264085\cdots $. S\'andor \cite{Sandor} improved this
to
$$e^{c\frac{k^3}{\log k}}\le K(k)\le c_0^{(1+\varepsilon)2^{k-1}},
\quad k\ge k_0.$$ The upper bound was recently improved by
Browning and Elsholtz \cite{BrowningandElsholtz} to
$$K(k)\le c_0^{(\frac{5}{48}+\varepsilon)2^k}, \quad k\ge k_0.$$

Finally, let us remark that the problem of representing 1 as a sum
of unit fractions with restricted prime factors in the
denominators is closely related to so called ``pseudoperfect''
numbers. A number is called pseudoperfect if it is the sum of some
of its divisors. For example, Sierpi\'{n}ski
\cite{Sierpinski:1965} observed that
\[ 945=315+189 +135+105+63+45+35+27+15+9+7\]
which is equivalent to a decomposition already stated by
Sierpi\'nski in \cite{Sierpinski}
\[1=\frac{1}{3}+\frac{1}{5}+ \frac{1}{7}
+ \frac{1}{9}+\frac{1}{15}+ \frac{1}{21}
+ \frac{1}{27}+\frac{1}{35}+ \frac{1}{63}
+ \frac{1}{105}+\frac{1}{135}. \]
Observe that the denominators consist of the prime factors 3, 5 and 7 only.

\section{Statement of results}

In this paper we prove the following results.

\begin{theorem}\label{thmto1} For $k\ge 4$ we have

$$T_o(2k+1)\ge   (\sqrt 2)^{(k+1)(k-4)}.$$ \end{theorem}

Let $p_1, p_2,\ldots , p_t$ be distinct primes. Define
$$K(p_1, p_2,\ldots , p_t)=\{ k : T_{k} (p_1, p_2,\ldots ,
p_t)\ge 1\} .$$ By Lemma \ref{lema},
if $k,l\in K(p_1, p_2,\ldots , p_t)$, then $k+l-1\in K(p_1, p_2,\ldots , p_t)$.
Observe that with $l\in K(p_1, p_2,\ldots , p_t)$,
the infinite arithmetic progression $a(l-1)+1$ is contained in
$K(p_1, p_2,\ldots , p_t)$.

\begin{theorem}\label{thma} Let $p_1, p_2,\ldots , p_t$ be distinct
primes. Then

(a) $K(p_1, p_2,\ldots , p_t)$ is a  union of finitely many
arithmetic progressions;

(b) there exist two constants $k_0=k_0(p_1, p_2,\ldots , p_t)$ and
$c_1=c_1(p_1, p_2,\ldots , p_t)>1$ such that for all $k>k_0$  with
$k\in K(p_1, p_2,\ldots , p_t)$ we have
$$c_1^k\le T_{k}(p_1, p_2,\ldots , p_t)\le \sqrt
2^{tk^2(1+o_k(1))}.$$\end{theorem}

Let
$$A=S(p_1,\ldots , p_t)\setminus \{ 1\} =\{ a_1<a_2<\cdots \} .$$
 Then
$$\sum_{i=1}^\infty \frac 1{a_i} =\left( 1+\frac 1{p_1}+\frac 1{p_1^2}+\cdots \right) \cdots \left( 1+\frac 1{p_t}+\frac 1{p_t^2}+\cdots
\right)-1=\frac{p_1}{p_1-1}\cdots \frac{p_t}{p_t-1}-1.$$
As we are studying finite sums of unit fractions, and
as the denominator 1 is discarded from consideration
 a necessary condition for
 $K(p_1, p_2,\ldots , p_t)$ to be  nonempty is:
\begin{equation}\label{>2}\frac{p_1}{p_1-1}\cdots
\frac{p_t}{p_t-1}>2.\end{equation} It is interesting that this necessary
condition \eqref{>2}
is also sufficient to guarantee that  $K(p_1, p_2,\ldots , p_t)$ is
nonempty.

\begin{theorem}\label{thmb}Let $p_1, p_2,\ldots , p_t$ be distinct primes.
Then $K(p_1, p_2,\ldots , p_t)$ is nonempty if and only if
$$\frac{p_1}{p_1-1}\cdots
\frac{p_t}{p_t-1}>2.$$

\end{theorem}

For a set $B$ of numbers, let
$$P(B)=\{ \sum_{a\in I} a \mid I\subseteq B, 0< |I|<\infty \} $$
denote the set of finite subset sums. For a set $B$ of nonzero
numbers, let
$$B^{-1}=\{ b^{-1} \mid  b\in B\} .$$

In order to prove Theorem \ref{thmb}, we make use of a well known
theorem of Graham \cite[Theorem 5]{Graham} and Birch \cite{Birch}
and observe, that $1$, or more generally $\frac{a}{b}$ can be
decomposed into a finite sum of distinct reciprocals for a more
general type of integer sequences. Graham's original hypotheses are
different, we adapt  his work for our applications. We prove the
following theorem.

\begin{theorem}\label{thmc} Let $A=\{ a_1<a_2<\cdots \} $ be a
sequence of positive integers such that

(a) $A$ is complete, i.e.~all sufficiently large integers are
contained in $P(A)$.

(b) $A$ is multiplicative, i.e.~for all $i,j$
with $a_i,a_j\in A$,  also $a_ia_j\in A$.

(c) $$\sum_{j=i+1}^\infty \frac 1{a_j} \ge \frac 1{a_i} , \quad
\text{ for all } i\ge 1.$$

Then $p/q\in P(A^{-1})$, where $(p, q)=1$, if and only if

(d) $$\frac pq <\sum_{i=1}^\infty \frac 1{a_i}$$ and

(e) $q$  divides some term of $A$.\end{theorem}

This implies the following corollary:

\begin{corollary}\label{cor1}Let $A=\{ a_1<a_2<\cdots \} $ be a
sequence of  integers with $a_1>1$ such that

(a) $A$ is complete;

(b) $A$ is multiplicative;

(c) $$\sum_{i=1}^\infty \frac 1{a_i} > 1 .$$

Then $1\in P(A^{-1})$.\end{corollary}

We pose the following problem for future research.

\begin{problem*} Let $p_1, p_2,\ldots , p_t$ be distinct primes. Is
there a constant $V$ depending only on $p_1, p_2,\ldots , p_t$
such that
$$T_{k}(p_1, p_2,\ldots ,
p_t)\le V^k ?$$\end{problem*}

Finally, we give two special results.

\begin{theorem}\label{thm1} (a) $T_k(3,5,7)\ge c_1\sqrt{62}^k$
for a computable constant $c_1>0$ and any odd number $k\ge 11$;

(b) $T_k(2,3,5)\ge c_2\sqrt{368}^k$ for a computable constant
$c_2>0$ and any integer $k\ge 3$.

\end{theorem}

\ \\

\section{Proof of Theorem \ref{thmto1}}

In order to prove Theorem \ref{thmto1}, we establish a
relation between $T_o(2k-1)$ and $T_o(2k+1)$, which inductively gives
a bound for an arbitrary odd number of fractions. For this purpose we
first establish the following lemma.

\begin{lemma}\label{T_o2} If $n$ is odd, then the number of solutions of
$$\frac 1n= \frac 1u+\frac 1v+\frac 1w, \quad n<u<v<w, 2\nmid uvw, d(w)\ge 2d(n)+1$$
is at least $\frac 12 d(n)-1$.\end{lemma}

\begin{proof}

Recall that the number of ways to write an integer $n$ as a sum of two
squares is  $r_2(n)=4(d_1(n)-d_3(n))$, where $d_i(n)$ is the number of positive
divisors $d$ of $n$ with $d\equiv i\pmod 4$ $(i=1, 3)$,
 (see \cite[Theorem 278 and
(16.9.2)]{Hardy} or \cite[Theorem 14.3]{Nathanson}):
As $r_2(n)$  is a non-negative integer
it follows that $d_1(n) \geq d_3(n)$ and
 $d_1(n) \geq \frac{1}{2}d(n)$.

 Let $k>1$ be a
positive factor of $n$ of the form $4l+1$. Let
$$u=n+2, \quad v=\frac 1{2k}n(n+2)(k+1), \quad w=\frac 12
n(n+2)(k+1).$$ Then
$$\frac 1n= \frac 1u+\frac 1v+\frac 1w, \quad n<u<v<w,\quad  2\nmid uvw.$$
Since $(k+1)/2>1$ is an integer and $(n, n+2)=1$, we have
$$d(w)=d(n(n+2)(k+1)/2)\ge d(n(n+2))+1=d(n)d(n+2)+1\ge 2d(n)+1.$$
This completes the proof of Lemma \ref{T_o2}.

\end{proof}

\begin{proof}[Proof of Theorem \ref{thmto1}]

 Let $T_o'(2k+1)$  denote the number of
solutions of $\sum_{i=1}^{2k+1}\frac 1{x_i}=1$ in odd numbers
$1<x_1<x_2<\cdots <x_{2k+1}$ with $d(x_{2k+1})>2^{k}$. Suppose
that $1<x_1<x_2<\cdots <x_{2k-1}(k\ge 5)$ is a solution of
$\sum_{i=1}^{2k-1}\frac 1{x_i}=1$ in odd numbers with
$d(x_{2k-1})> 2^{k-1}$.  By Lemma \ref{T_o2} the number of
solutions of
$$\frac 1{x_{2k-1}}= \frac 1u+\frac 1v+\frac 1w, \quad x_{2k-1}<u<v<w, 2\nmid uvw, d(w)\ge 2d(x_{2k-1})+1$$
is at least $\frac 12 d(x_{2k-1})-1$. Since
$$d(w)\ge 2d(x_{2k-1})+1>2^{k},\quad \frac 12 d(x_{2k-1})-1\ge \frac
12 (2^{k-1}+1)-1=2^{k-2}-\frac 12,$$ we have that the number of
solutions of
$$\frac 1{x_{2k-1}}= \frac 1u+\frac 1v+\frac 1w, \quad x_{2k-1}<u<v<w, 2\nmid uvw, d(w)> 2^{k}$$
is at least $2^{k-2}$. Hence
$$T_o'(2k+1)\ge 2^{k-2}T_o'(2k-1).$$
By Shiu \cite{Shiu} (see also \cite{Burshtein}) there exist $9$
odd numbers $1<x_1<x_2<\cdots <x_{9}$ with $x_{9}=10395$ and
$$\sum_{i=1}^{9}\frac 1{x_i}=1.$$
Since $d(10395)=32$, we have  $T_o'(9)\ge 1$. Thus
$$T_o'(2k+1)\ge 2^{k-2}T_o'(2k-1)\ge \cdots \ge
2^{(k-2)+(k-3)+\cdots +(5-2)}T_o'(9)\ge 2^{\frac 12(k+1)(k-4)}.$$
Hence
$$T_o(2k+1)\ge (\sqrt 2)^{(k+1)(k-4)}.$$
This completes the proof of Theorem \ref{thmto1}.
\end{proof}

\section{Proof of Theorem \ref{thma}}\label{proofthma}

For distinct primes $p_1, p_2,\ldots , p_t$, let $\mathcal{T}_{k}
(p_1, p_2,\ldots , p_t)$ be the set of all solutions $(x_1,\ldots
, x_k)$ of $$ \sum_{i=1}^k \frac 1{x_i}=1, 1<x_1<x_2<\cdots <x_k,
x_i\in S(p_1, p_2,\ldots , p_t).$$ Define
$$(x_1,\ldots , x_k)\ast (y_1,\ldots , y_l)=(x_1,\ldots , x_{k-1}, x_ky_1, \ldots ,
x_ky_l)$$ and
$$(a_1, \ldots , a_k)^i=(a_1,\ldots , a_k)^{i-1}\ast (a_1, \ldots ,
a_k),\quad i\ge 2.$$ It is clear that if $(x_1,\ldots , x_k)\in
\mathcal{T}_{k} (p_1, p_2,\ldots , p_t)$ and $(y_1,\ldots ,
y_l)\in \mathcal{T}_{l} (p_1, p_2,\ldots , p_t)$, then
\begin{equation}\label{eqnws}(x_1,\ldots , x_k)\ast (y_1,\ldots , y_l)\in
\mathcal{T}_{k+l-1} (p_1, p_2,\ldots , p_t).\end{equation}

The following lemma gives a recursive lower bound:

\begin{lemma}\label{lema}  Let $p_1, p_2,\ldots , p_t$ be distinct primes.
Then, for any two positive integers $k$ and $l$, we have
$$T_{k+l-1} (p_1, p_2,\ldots , p_t)\ge T_{k} (p_1, p_2,\ldots , p_t)T_{l} (p_1, p_2,\ldots ,
p_t).$$\end{lemma}

\begin{proof} We define a map $f$ from $\mathcal{T}_{k} (p_1, p_2,\ldots , p_t)\times \mathcal{T}_{l} (p_1, p_2,\ldots ,
p_t)$ to $\mathcal{T}_{k+l-1} (p_1, p_2,\ldots , p_t)$ as follows:
$$(x_1,\ldots , x_k)\times (y_1,\ldots , y_l) \longmapsto
(x_1,\ldots , x_k)\ast (y_1,\ldots , y_l).$$ It is clear that $f$
is injective. Now Lemma \ref{lema} follows immediately.
\end{proof}

\begin{lemma}\label{lemd} Let $p_1, p_2,\ldots , p_t$ be distinct
primes. If $(x_1,\ldots , x_k)\in \mathcal{T}_{k} (p_1, p_2,\ldots
, p_t)$ and $(y_1,\ldots , y_l)\in \mathcal{T}_{l} (p_1,
p_2,\ldots , p_t)$ with $x_k^{l-1}\not= y_l^{k-1}$, then
$$T_{(k-1)(l-1)+1} (p_1, p_2,\ldots , p_t)\ge 2.$$\end{lemma}

\begin{proof}
 By (\ref{eqnws}) we have
$$(x_1,\ldots , x_k)^{l-1}, (y_1,\ldots , y_l)^{k-1} \in \mathcal{T}_{(k-1)(l-1)+1} (p_1, p_2,\ldots ,
p_t).$$ Since  $x_k^{l-1}$ and $ y_l^{k-1}$ are the largest elements
of $(x_1,\ldots , x_k)^{l-1}$ and $(y_1,\ldots , y_l)^{k-1}$
respectively, by $x_k^{l-1}\not= y_l^{k-1}$ we have
$$(x_1,\ldots , x_k)^{l-1}\not= (y_1,\ldots , y_l)^{k-1}.$$
Hence $T_{(k-1)(l-1)+1} (p_1, p_2,\ldots , p_t)\ge 2$. This
completes the proof of Lemma \ref{lemd}.

\end{proof}

The following lemma is an extension of a well known theorem of Birch
\cite{Birch}. The possibility for this extension was already
mentioned by Davenport and Birch (see \cite{Birch} and
\cite{Hegyvari}). Hegyv\'ari \cite{Hegyvari} gave an explicit value
of $K(p,q)$. The upper bound of $K(p,q)$ was improved recently by
Fang \cite{fang}.

\begin{lemma}\label{lemb}(Hegyv\'ari \cite{Hegyvari}) For every
 integers $p,q$ with $p>1$, $q>1$ and $(p,q)=1$,   there exists
 $K=K(p,q)$ such that the set
 $$Y_K=\{ p^\alpha q^\beta \mid \alpha , \beta \in \mathbf{N}_0 ,
 0\le \beta \le K\} $$
 is complete. That is, every sufficiently large integer is the sum of distinct terms taken from $Y_K$. \end{lemma}

\begin{lemma}\label{lemc}
Let $p_1, p_2,\ldots , p_t$ be distinct primes. If $T_{k} (p_1,
p_2,\ldots , p_t)\ge 1$ for some $k$, then  $T_{l} (p_1,
p_2,\ldots , p_t)\ge 2$ for some $l$.\end{lemma}

\begin{proof} Let $(x_1,\ldots , x_k)\in \mathcal{T}_{k} (p_1, p_2,\ldots
, p_t)$. It is clear that $x_k$ is not a prime power. Therefore, there
exist two distinct primes $p,q\in \{ p_1, p_2,\ldots , p_t \} $
with $pq|x_k$. Let $K$ be as in Lemma \ref{lemb}. Take a large
$v>K$ such that $q^v$ is the sum of distinct terms taken from
$Y_K$.  Assume that
$$q^v=\sum_{i=1}^t p^{\alpha_i}q^{\beta_i}, \quad
p^{\alpha_1}q^{\beta_1}<\cdots < p^{\alpha_t}q^{\beta_t} ,$$ where
$\alpha_i, \beta_i\in \mathbf{N}_0$ and $0\le \beta_i\le K$. Since
$v>K$, we have $t\ge 2$ and $v>\beta_i (1\le i\le t)$. Let $u=\max
\{ v, \alpha_1, \ldots , \alpha_t\} $. Write
$$(x_1,\ldots , x_k)^u =(y_1, \ldots , y_{u(k-1)+1}).$$
Then $y_{u(k-1)+1}=x_k^u$. It is clear that
$$(y_1, \ldots , y_{u(k-1)},
y_{u(k-1)+1}q^{v-\beta_t}p^{-\alpha_t},\ldots ,
y_{u(k-1)+1}q^{v-\beta_1}p^{-\alpha_1})\in \mathcal{T}_{u(k-1)+t}
(p_1, p_2,\ldots , p_t).$$ In order to prove Lemma \ref{lemc}, by
Lemma \ref{lemd} it is enough to prove that
$$y_{u(k-1)+1}^{u(k-1)+t-1}\not=\Big(
y_{u(k-1)+1}q^{v-\beta_1}p^{-\alpha_1}\Big)^{u(k-1)},$$ or equivalently
$$y_{u(k-1)+1}^{t-1}p^{u(k-1)\alpha_1}\not=
q^{u(k-1)(v-\beta_1)}.$$ This follows from $t\ge 2$,
$u(k-1)\alpha_1\ge 0$ and
$$pq\mid y_{u(k-1)+1}^{t-1}.$$
 This completes the proof of Lemma \ref{lemc}.

\end{proof}

\begin{proof}[Proof of Theorem \ref{thma}]  If $K(p_1, p_2,\ldots , p_t)$
is empty, then
Theorem \ref{thma} is true trivially.  Now we assume that $K(p_1,
p_2,\ldots , p_t)$ is not empty.

We first proof part (a).
By Lemma \ref{lemc} there exists
an $m_0$ with $T_{m_0} (p_1, p_2,\ldots , p_t)\ge 2$. For each
integer $0\le i <m_0-1$,  let $k_i$ (if it exists) be the least
positive integer $k$ such that $k\equiv i\pmod{m_0-1} $ and $k\in
K(p_1, p_2,\ldots , p_t)$. By Lemma \ref{lema} we have
\begin{equation}\label{eqnab} T_{(m_0-1)l+k_i}(p_1, p_2,\ldots , p_t)\ge (T_{m_0}(p_1,
p_2,\ldots , p_t))^l T_{k_i}(p_1, p_2,\ldots , p_t)\ge
2^l.\end{equation} Hence
\begin{equation}\label{new}K(p_1, p_2,\ldots , p_t)=\bigcup_{i=0,\, k_i
\text{exists}}^{m_0-2} \{ (m_0-1)l+k_i : l=1, 2, \ldots \},
\end{equation}
which proves part (a).

We now prove the lower bound of
part (b) let $c_1=\min 2^{1/(m_0-1+k_i)}$ and $t_0=\max
k_i$, where the minimum and maximum are taken over all $i$ such
that $k_i$ exists.  If $k\in K(p_1, p_2,\ldots , p_t)$ and
$k>t_0$, then, by \eqref{new},  there exists an $i$ with $0\le i
<m_0-1$ and a positive integer $l$  such that $k=(m_0-1)l+k_i$.
  By (\ref{eqnab}) we have
$$T_{k}(p_1, p_2,\ldots , p_t)=T_{(m_0-1)l+k_i}(p_1, p_2,\ldots , p_t)
\ge 2^l\ge 2^{((m_0-1)l+k_i)/(m_0-1+k_i)} \ge c_1^k.$$

To prove the upper bound
let
 $$\sum_{i=1}^{k}\frac 1{x_i}=1, \quad 1<x_1<x_2<\cdots <x_{k},
\quad x_i\in S(p_1, p_2,\ldots , p_t)\ (1\le i\le k).$$

Let the sequence $u_n$ be defined in the following way: $u_1=1,
u_{n+1}=u_n(u_n+1)$ for $n\ge 1$. Then $u_n<2^{2^n}$ for $n\ge 1$.
As in the proof of \cite[p.~218]{Sandor} we have
$$x_j\le (k-j+1)u_j<k 2^{2^j},\quad j=1,2,\ldots , k.$$
Let
$$x_j=p_1^{\alpha_{j1}}\cdots p_t^{\alpha_{jt}}.$$
Then $\alpha_{ji}\le 2\log k+2^j$. Thus
$$T_k(p_1,\ldots , p_t)\le \prod_{2^j\le 2\log k} (4\log k)^t
\prod_{j\le k, 2^j>2\log k} (2^{j+1})^t =\sqrt 2^{tk^2(1+o(1))}.$$
This completes the proof of Theorem \ref{thma}.
\end{proof}

\section{Proofs of Theorems \ref{thmb}, \ref{thmc} and Corollary \ref{cor1}}

In order to prove Theorem \ref{thmc}, we need a well known result
of Graham.

For a sequence $S=(s_1, s_2, \ldots ) $ of positive integers,
$M(S)$ is defined to be the monotonically increasing sequence formed
from the set of all products $\prod\limits_{i=1}^m s_{k_i}$, where
$m=1, 2,\ldots $ and $k_1<k_2<\cdots <k_m$. Thus all the terms of
$M(S)$ are distinct.

For a sequence of real numbers, a real number $\alpha $ is said to
be $S$-\emph{accessible} if, for any $\varepsilon >0$, there
exists $\beta \in P(S)$ such that $0\le \beta -\alpha <\varepsilon
$.

$S$ is said to be \emph{complete} if all sufficiently large
integers belong to $P(S)$.\\

{\bf Theorem A (\cite[Theorem 5]{Graham})} {\it Let $S=(s_1, s_2,
\ldots ) $ be a sequence of positive integers such that

(1) $M(S)$ is complete,

(2) $s_{n+1}/s_n$ is bounded.

Then
$$\frac pq\in P((M(S))^{-1})$$
(where $(p, q)=1$) if and only if

(3) $p/q$ is $(M(S))^{-1}$-accessible,

(4) $q$ divides some term of $M(S)$.}

With these preparation, we can prove our Theorem \ref{thmc}.\\

\begin{proof}[Proof of Theorem \ref{thmc}]
By (b) we have $M(A)=A$. By (a) and $M(A)=A$ we know that Theorem
A (1) is true. By (b) we have $a_2a_n\in A$.   As $a_2>a_1\ge 1$
we have $a_2a_n>a_n$. Thus $a_{n+1}\le a_2a_n$. So $a_{n+1}/a_n\le
a_2$. Hence Theorem A (2) is true.

If $p/q\in P(A^{-1})$, where $(p, q)=1$, then (d) is true and by
Theorem A, we have that $q$  divides some term of $A$, i.e. (e) is true.

Now we assume that (d) and (e) are true. From (e) we know that
Theorem A (4) holds.  In order to prove that $p/q\in P(A^{-1})$,
by Theorem A, it is enough to prove that Theorem A (3) holds,
i.e., $p/q$ is $A^{-1}$-accessible.

Suppose that $p/q\notin P(A^{-1})$ (this avoids equality
 in the following arguments). We will show that $p/q$ is
$A^{-1}$-accessible. Then by Theorem A we have $p/q\in P(A^{-1})$,
a contradiction.

If
$$\sum_{i=1}^\infty \frac 1{a_i}=+\infty,$$
let $a_0=0$. If
$$\sum_{i=1}^\infty \frac 1{a_i}<+\infty,$$
let $a_0$ be the real number defined by
$$\frac 1{a_0}=\sum_{i=1}^\infty \frac 1{a_i}.$$

Let $i_1$ be the integer $i$ such that
\begin{equation}\label{eqnu}\frac 1{a_i}<\frac pq<\frac 1{a_{i-1}}.\end{equation} By (d) we
have $i_1\ge 1$. By (c) and (d) we have
\begin{equation}\label{eqnv}\sum_{i=i_1}^\infty \frac 1{a_i}\ge \frac 1{a_{i_1-1}}>\frac
pq.\end{equation} By (\ref{eqnu}) and (\ref{eqnv}) we have
$$0<\frac pq-\frac 1{a_{i_1}}<\frac 1{a_{i_1-1}}-\frac 1{a_{i_1}}\le \sum_{i=i_1+1}^\infty \frac 1{a_i}.$$
Suppose that we have found a sequence $\{ i_k\}_{k=1}^n$ such that
$1\le i_1<i_2<\cdots <i_n $ and
$$0<\frac pq-\sum_{l=1}^k\frac 1{a_{i_l}}< \sum_{i=i_k+1}^\infty \frac 1{a_i}, \quad k=1,2,\ldots, n .$$
If
$$\frac 1{a_{i_n+1}}<\frac pq-\sum_{l=1}^n\frac 1{a_{i_l}},$$
let $i_{n+1}=i_n+1$, then
$$0<\frac pq-\sum_{l=1}^{n+1}\frac 1{a_{i_l}}< \sum_{i=i_{n+1}+1}^\infty \frac 1{a_i}.$$
If
$$\frac pq-\sum_{l=1}^n\frac 1{a_{i_l}}<\frac 1{a_{i_n+1}},$$
let $i_{n+1}$ be the integer $i$ with
$$\frac 1{a_i}<\frac pq-\sum_{l=1}^n\frac 1{a_{i_l}}<\frac
1{a_{i-1}},$$ then $i_{n+1}>i_n+1$ and
$$0<\frac pq-\sum_{l=1}^{n+1}\frac 1{a_{i_l}}<\frac
1{a_{i_{n+1}-1}}- \frac 1{a_{i_{n+1}}} \le \sum_{i=i_{n+1}+1}^\infty
\frac 1{a_i}.$$ Thus we can find a sequence $\{ i_k\}_{k=1}^\infty$
such that $1\le i_1<i_2<\cdots  $ and
$$0<\frac pq-\sum_{l=1}^k\frac 1{a_{i_l}}< \sum_{i=i_k+1}^\infty \frac 1{a_i}, \quad k=1,2,\ldots .$$
Let $j_k$ be the least $j$ with $j\ge i_k+1$ such that
$$0<\frac pq-\sum_{l=1}^k\frac 1{a_{i_l}}< \sum_{i=i_k+1}^j \frac 1{a_i}.$$
Then
$$0<\sum_{i=i_k+1}^{j_k} \frac 1{a_i}-\left( \frac pq-\sum_{l=1}^k\frac 1{a_{i_l}}
\right) < \frac 1{a_{j_k}}.$$ That is
$$0<\sum_{l=1}^k\frac 1{a_{i_l}}+\sum_{i=i_k+1}^{j_k}\frac 1{a_i}-\frac pq<\frac 1{a_{j_k}}.$$
Since $a_{j_k}\to\infty$, we have $p/q$ is $A^{-1}$-accessible.
This completes the proof of Theorem \ref{thmc}.

\end{proof}

\begin{proof}[Proof of Corollary \ref{cor1}] By Theorem \ref{thmc}
it is enough to prove that
$$\sum_{j=i+1}^\infty \frac 1{a_j} \ge \frac 1{a_i}  \quad
\text{ for all } i\ge 1.$$ Since
$$a_i<a_ia_1<a_ia_2<\cdots ,$$
we have
$$\sum_{j=i+1}^\infty \frac 1{a_j} \ge \sum_{j=1}^\infty \frac
1{a_ia_j}>\frac 1{a_i}.$$ This completes the proof of Corollary
\ref{cor1}.

\end{proof}

\begin{proof}[Proof of Theorem \ref{thmb}] Let
$$A=S(p_1,\ldots , p_t)\setminus \{ 1\} =\{ a_1<a_2<\cdots \} .$$
The necessity of the condition was explained as motivation just
before the statement of Theorem \ref{thmb}. We only need to prove
the sufficiency. Assume that
$$\frac{p_1}{p_1-1}\cdots \frac{p_t}{p_t-1}>2.$$
Then $t\ge 2$ and
\begin{equation}\label{eqnw}\sum_{i=1}^\infty \frac 1{a_i}>1.\end{equation}
Since $t\ge 2$, by Lemma \ref{lemb}, $A$ is complete. It is clear
that (b) in Corollary \ref{cor1} is true.  By Corollary \ref{cor1}
we have $1\in P(A^{-1})$. This completes the proof of Theorem
\ref{thmb}.
\end{proof}

\section{Proof of Theorem \ref{thm1}}

Let $A_k(M)$ denote the set of solutions of $\sum_{i=1}^k\frac
1{x_i}=1$ in distinct integers $1<x_1<x_2<\cdots <x_k$ with
$M|x_k$ and $x_i\in\{ 2^\alpha 3^\beta 5^\gamma \}\ (1\le i\le
k)$. Let $B_k(M)$ denote the set of solutions of
$\sum_{i=1}^k\frac 1{x_i}=1$ in distinct integers
$1<x_1<x_2<\cdots <x_k$ with $M|x_k$ and $x_i\in\{ 3^\alpha
5^\beta 7^\gamma \}\ (1\le i\le k)$. It is clear that $T_k(2, 3,
5)\ge |A_k(M)|$ for any $M\in \{ 2^\alpha 3^\beta 5^\gamma \}$ and
$T_k(3, 5, 7)\ge |B_k(M)|$ for any $M\in \{ 3^\alpha 5^\beta
7^\gamma \}$. In order to obtain  good lower bounds of $T_k(2, 3,
5)$ and $T_k(3, 5, 7)$, we choose two suitable constants $M_1$ and $M_2$
such that $|A_k(M_1)|$ and $|B_k(M_2)|$ have good lower bounds to start with.
 We will establish recursive relations between $|A_{k+2}(M)|$ and
$|A_k(M)|$, and $|B_{k+2}(M)|$ and $|B_k(M)|$ which inductively proves the
desired result. Observe that $B_{2k}(M)=\emptyset $).

We start with the following lemma.

\begin{lemma}\label{lem2} Let $m_i$, $a_i,b_i,c_i,d_i$ be  nonzero integers
with $0<a_i<b_i<c_i$, $a_i+b_i+c_i=d_i$ and $(a_i,b_i,c_i)=1$
$(i=1,2)$. If
$$\left\{
\frac{ d_1m_1}{a_1},\frac{d_1m_1}{b_1},\frac{d_1m_1}{c_1}\right\}
=\left\{ \frac{
d_2m_2}{a_2},\frac{d_2m_2}{b_2},\frac{d_2m_2}{c_2}\right\} ,$$
then $a_1=a_2$, $b_1=b_2$, $c_1=c_2$ and $m_1=m_2$.
\end{lemma}

 \begin{proof} Since $0<a_i<b_i<c_i\ (i=1,2)$ and
$$\left\{
\frac{ d_1m_1}{a_1},\frac{d_1m_1}{b_1},\frac{d_1m_1}{c_1}\right\}
=\left\{
\frac{d_2m_2}{a_2},\frac{d_2m_2}{b_2},\frac{d_2m_2}{c_2}\right\}
,$$ we have
$$\frac{ d_1m_1}{a_1}=\frac{
d_2m_2}{a_2},\quad\frac{d_1m_1}{b_1}=\frac{d_2m_2}{b_2},\quad
\frac{d_1m_1}{c_1}=\frac{d_2m_2}{c_2}.$$ Thus
$$a_2d_1m_1=a_1d_2m_2,\quad b_2d_1m_1=b_1d_2m_2,\quad
c_2d_1m_1=c_1d_2m_2.\eqno{(1)}$$ Hence
$$(a_2d_1m_1,b_2d_1m_1,c_2d_1m_1)=(a_1d_2m_2,b_1d_2m_2,c_1d_2m_2).$$
Since $(a_i,b_i,c_i)=1\ (i=1,2)$, we have $d_1m_1=d_2m_2$. By (1)
we have $a_1=a_2$, $b_1=b_2$ and $c_1=c_2$. Thus $d_1=d_2$. By
$d_1m_1=d_2m_2$ we have $m_1=m_2$. This completes the proof of
Lemma \ref{lem2}.\end{proof}

The following two lemmas establish the recursive relations
between $|A_{k+2}(M)|$ and $|A_k(M)|$, and
$|B_{k+2}(M)|$ and $|B_k(M)|$.\\

\begin{lemma}\label{lem3} Let $M_2=3^{20}\times 5^{20}\times 7^{20}$. Then
$$|B_{k+2}(M_2)|\ge 62|B_k(M_2)|.$$
\end{lemma}

\begin{proof} If $|B_k(M_2)|=0$, then the conclusion is clear.  Now
we assume that $|B_k(M_2)|>0$. By Lemma \ref{lem2} we only need to
find 62 four-tuples $(a, b, c, d)$ to each $(x_1, \ldots , x_k)\in
B_k(M_1)$ with $a,b,c,d\in \{ 3^\alpha 5^\beta 7^\gamma \}$,
$a+b+c=d$, $a<b<c$, $(a,b,c)=1$,$a|dx_k$, $b|dx_k$,$c|dx_k$ and
$M_1|\frac{dx_k}{a}$. The reason is that
$$\frac 1{x_k}=\frac 1{dx_k/c}+\frac 1{dx_k/b}+\frac 1{dx_k/a}$$
and
$$x_k<dx_k/c<dx_k/b<dx_k/a.$$
By a simple computer program with Mathematica we find that there
are 62 four-tuples  $(a, b, c, d)$ with $a,b,c,d\in \{ 3^\alpha
5^\beta 7^\gamma : 0\le \alpha\le 14,
 0\le \beta, \gamma\le 8\}$, $a+b+c=d$, $a<b<c$, $(a,b,c)=1$ and
$a|d$. Since $M_2=3^{20}\times 5^{20}\times 7^{20}$ and $M_2|x_k$,
Lemma \ref{lem3}  follows immediately.\end{proof}

\begin{lemma}\label{lem4} Let $M_1=2^{20}\times 3^{20}\times 5^{20}$. Then
$$|A_{k+2}(M_1)|\ge 368|A_k(M_1)|.$$
\end{lemma}

\begin{proof}
By a simple computer program with Mathematica we find that there
are 368 four-tuples  $(a, b, c, d)$ with $a,b,c,d\in \{ 2^\alpha
3^\beta 5^\gamma : 0\le \alpha\le 15,
 0\le \beta\le 10, 0\le \gamma\le 8\}$, $a+b+c=d$, $a<b<c$, $(a,b,c)=1$ and
$a|d$. Similarly to the proof of  Lemma \ref{lem3}, Lemma
\ref{lem4} follows immediately.\end{proof}

Eventually, we come to the proof of Theorem 1.

\begin{proof}[\bf Proof of Theorem 1.] (a) By Sierpi\'nski
\cite{Sierpinski} (see also \cite{Burshtein1}) there exist $11$
odd numbers $1<x_1<x_2<\cdots <x_{11}$ with $x_{11}=945$,
$x_i\in\{ 3^\alpha 5^\beta 7^\gamma \}$ and
$$\sum_{i=1}^{11}\frac 1{x_i}=1,$$
(see the introduction.)
Since
$$\frac 1{x_{11}}=\frac 1{105x_{11}/(3^2\times 7)}+\frac 1{105x_{11}/3^3}+\frac
1{105x_{11}/3^2}+\frac 1{105x_{11}/5}+\frac 1{105x_{11}},$$ there
exist $15$ odd numbers $1<y_1<y_2<\cdots <y_{15}$ with
$y_{15}=105x_{11}$,$y_i\in\{ 3^\alpha 5^\beta 7^\gamma \}$ and
$$\sum_{i=1}^{15}\frac 1{y_i}=1.$$
Continuing this procedure,  there exists an odd number $k_0$ such
that $|B_{k_0}(M_2)|\ge 1$. By Lemma \ref{lem3} we have
$$|B_k(M_2)|\ge \sqrt{62}^{k-k_0}|B_{k_0}(M_2)|\ge
\sqrt{62}^{k-k_0},\quad k\ge k_0, 2\nmid k.$$ So
$$T_k(3,5,7)\ge \sqrt{62}^{k-k_0},\quad k\ge k_0, 2\nmid k.$$
Since
$$\frac 1{x_{11}}=\frac 1{5x_{11}/3}+\frac 1{3x_{11}}+\frac
1{15x_{11}},$$  there exist $13$ odd numbers $1<z_1<z_2<\cdots
<z_{13}$ with $z_{13}=15x_{11}$,$z_i\in\{ 3^\alpha 5^\beta
7^\gamma \}$ and
$$\sum_{i=1}^{13}\frac 1{z_i}=1.$$
Continuing this procedure, we have $T_k(3,5,7)\ge 1$ for all odd
numbers $k\ge 11$.  Hence there exists a positive constant $c_1$
such that $T_k(3,5,7)\ge c_1\sqrt{62}^k$ for all odd numbers $k\ge
11$.

(b) By $$\frac 12 +\frac 1{4}+\frac 1{8}+\frac 1{16}+\frac
1{30}+\frac 1{60}+\frac 1{120}+\frac 1{240} =1, $$
$$\frac 1a=\frac 1{30a/24}+\frac 1{30a/3}+\frac 1{30a/2}+\frac
1{30a}$$ and
$$\frac 1a=\frac 1{3a/2}+\frac 1{3a},$$
there exists an integer $k_0$ such that
$$|A_{2k_0}(M_1)|\ge 1,\quad |A_{2k_0+1}(M_1)|\ge 1.$$ By Lemma
\ref{lem4} we have
$$|A_{2k}(M_1)|\ge \sqrt{368}^{2k-2k_0}|A_{2k_0}(M_1)|\ge
{368}^{k-k_0},\quad k\ge k_0$$ and
$$|A_{2k+1}(M_1)|\ge \sqrt{368}^{2k-2k_0}|A_{2k_0+1}(M_1)|\ge
{368}^{k-k_0},\quad k\ge k_0.$$ By $$\frac 12 +\frac 13+\frac
1{3^2}+\cdots +\frac 1{3^k}+\frac 1{2\times 3^k} =1, $$ we have
$T_k(2,3,5)\ge 1$ for all $k\ge 3$. So there exists a positive
constant $c_2$ such that
 $T_k(2,3,5)\ge c_2\sqrt{368}^k$ for all integers $k\ge
3$.

This completes the proof of Theorem \ref{thm1}.\end{proof}

{\bf Acknowledgements. }  We are grateful to the referee for
his/her sincerely and helpful comments.\\

\end{document}